\documentclass[]{amsart} 

\usepackage{amsmath,amsthm,amssymb,amsfonts,mathrsfs,color,hyperref, mathtools,crop, graphicx, enumitem}
\usepackage[]{geometry}

\theoremstyle{plain}
\begingroup
\newtheorem{theorem}{Theorem}[section]
\newtheorem*{theorem*}{Theorem}
\newtheorem*{"theorem"}{``Theorem''}

\newtheorem{lemma}[theorem]{Lemma}
\endgroup

\theoremstyle{definition}
\begingroup
\newtheorem{definition}[theorem]{Definition}
\endgroup

\theoremstyle{remark}
\begingroup
\newtheorem{remark}[theorem]{Remark}

\endgroup 

\numberwithin{equation}{section}
\newcommand{\N}{\mathbb N}

\newcommand{\R}{\mathbb R} 

\newcommand{\dist}{{\rm dist}}

\renewcommand{\div}{{\rm div}}
\newcommand{\sdist}{{\mathrm{sdist}}}

\newcommand{\id}{\mathrm{id}}

\newcommand{\Per}{\mathrm{Per}}

\renewcommand{\H}{{\mathcal H}}
\newcommand{\E}{{\mathcal E}}

\newcommand{\C}{{\mathcal C}}
\newcommand{\F}{{\mathcal F}}

\newcommand{\Ra} {\Rightarrow}

\renewcommand{\d}{\,\mathrm{d}}

\newcommand{\dx}{\,\mathrm{d}x}
\newcommand{\dy}{\,\mathrm{d}y}

\newcommand{\ds}{\,\mathrm{d}s}
\newcommand{\dt}{\,\mathrm{d}t}

\newcommand{\cc}{\Subset}

\newcommand{\eps}{\varepsilon}
\newcommand{\average}{{\mathchoice {\kern1ex\vcenter{\hrule height.4pt
width 6pt depth0pt} \kern-9.7pt} {\kern1ex\vcenter{\hrule
height.4pt width 4.3pt depth0pt} \kern-7pt} {} {} }}

\newcommand{\Lto}{\stackrel{L^1}\longrightarrow}
\newcommand{\conv}{\mathrm{conv}}

\graphicspath{{figs/}}

\begin{document}

\title[A phase-field model for connected perimeters]{Approximation of the relaxed perimeter functional under a connectedness constraint by phase-fields}

\author{Patrick W.~Dondl}
\address{Patrick W.~Dondl\\Abteilung f\"ur Angewandte Mathematik\\Albert-Ludwigs-Universit\"at Freiburg\\Hermann-Herder-Str.~10\\ 79104 Freiburg i.~Br.\\Germany}
\email{patrick.dondl@mathematik.uni-freiburg.de}

\author{Matteo Novaga}
\address{Matteo Novaga\\ Dipartimento di Matematica Universit\`a di Pisa\\ Largo Bruno Pontecorvo 5\\ 56127 Pisa\\ Italy}
\email{matteo.novaga@unipi.it}

\author{Benedikt Wirth}
\address{Benedikt Wirth\\ Institut f\"ur Analysis und Numerik\\ Westf\"alische Wilhelms-Universit\"at M\"unster\\ Einsteinstr. 62\\ 48149 M\"unster\\ Germany}
\email{Benedikt.Wirth@uni-muenster.de }

\author{Stephan Wojtowytsch}
\address{Stephan Wojtowytsch\\Department of Mathematical Sciences\\Carnegie-Mellon University\\5000 Forbes Avenue\\ Pittsburgh, PA 15213, USA}
\email{swojtowy@andrew.cmu.edu}

\date{\today}


\begin{abstract}
We develop a phase-field approximation of the relaxation of the perimeter functional in the plane under a connectedness constraint based on the classical Modica-Mortola functional and the connectedness constraint of \cite{MR3590663}. We prove convergence of the approximating energies and present numerical results and applications to image segmentation.
\end{abstract}

\maketitle

\tableofcontents

\section{Introduction}
In this article, we consider a phase field approximation of the problem of finding a ``connected perimeter'' of a set $E\subset\R^2$. This connected perimeter is given as the limit of the perimeter of an optimal sequence of connected sets approximating $E$ in an $L^1$-sense. An application of the phase field energy developed here may be the segmentation of a given image to yield a connected (or simply connected) set. Our functional is based on the classical Modica-Mortola energy with an additional energy term that penalizes non-path-connectedness of the preimage of a given interval under the phase field function. Similar to the methods in~\cite{bonnivard:2014tw, Benmansour:2010dm} we use a geodesic distance in order to detect path-(dis)connectedness.

In~\cite{MR3590663, Dondl:2018wb}, this topological functional was introduced in the context of diffuse curvature dependent energies. Here, we show that the $\Gamma$-limit of the sum of the usual Modica-Mortola-energy and our topological energy~\eqref{eq:top_energy} is given by the $L^1$-relaxation of the perimeter functional when only considering connected approximating sets. In the sharp-interface setting, this relaxation has been studied in~\cite{novaga-etal}.

The article is structured as follows. In section \ref{section review sharp interface} we recall the results of \cite{novaga-etal} for the sharp interface problem and construct the connection between our result on finite domains and the sharp interface problem posed in the plane. Section \ref{section phase field functional} contains an intuitive explanation of how our phase-field energy incorporates the connectedness constraint. The main result on the approximation of the connected relaxation of the perimeter functional by phase-field energies is stated and proved in section \ref{section main}. Some extensions of the result concerning approximation by simply connected sets and relative perimeters are collected in section \ref{section extensions}. Finally, we present numerical evidence for the effectiveness of our approach to connectedness for diffuse sets in section \ref{section numerics}. A technical result on the approximation of the closest point projection onto a closed convex sets by non-expansive diffeomorphisms is proved in the appendix.

\section{Approximation by Connected Sets}

\subsection{The sharp interface model}\label{section review sharp interface}

For open sets $E\subset \R^2$ such that the characteristic function $\chi_E$ of $E$ is in $BV(\R^2)$, the perimeter functional
\[
P(E) = |D\chi_E|(\R^2) = \sup\left\{ \int_E \div(u)\dx \:\bigg|\: u\in C_c^1(\R^2), \:|Du|\leq 1\right\}
\]
is a generalized measure of the size of the boundary of $E$ which agrees with the $\H^1$-measure of the boundary on Lipschitz sets due to the Gauss-Green theorem. It is well known that $P$ is lower semi-continuous under the $L^1$-convergence of characteristic functions \cite[Proposition 3.38]{ambrosio2000functions}, and that the characteristic functions of $C^\infty$-smooth open sets lie $L^1$-dense in the collection of characteristic functions of sets of finite perimeter \cite[Theorem 3.42]{ambrosio2000functions}. Therefore, the $L^1$-lower semi-continuous envelope of the perimeter functional without any additional constraints agrees with the functional itself, independently of whether the approximating sets in the relaxation process are taken to be smooth or not. In this article, we wish to consider a similar relaxation process, but under additional topological constraints.

For an open set $E\subset \R^2$ we define two relaxations of the perimeter functional under a connectedness constraint
\begin{align*}
\overline{P_{C}}(E) & : = \inf\left\{\liminf_{n\to \infty} P(E_n)\:\bigg|\: E_n\stackrel{L^1}\to E, \:E_n \text{ indecomposable} \right\}\\
\overline{P_{C}^r}(E) & : = \inf\left\{\liminf_{n\to \infty} P(E_n)\:\bigg|\: E_n\stackrel{L^1}\to E,  \:E_n \text{ connected and $C^\infty$-smooth} \right\}
\end{align*}
which differ in the degree of smoothness required of the approximating sets. Here `indecomposable' is a measure-theoretic analogue of the notion of connectedness for open sets which are only defined in the $L^1$-sense. As usual, we take the $L^1$-topology on equivalence classes of bounded measurable sets as induced by the metric given by the $L^1$-distance of their characteristic functions or equivalently the Lebesgue measure of their symmetric difference $E\Delta F = (E\cup F) \setminus E\cap F$
\[
d_{L^1}(E, F) = \int_{\R^2} |\chi_E - \chi_F| \dx = \big|E\Delta F\big|.
\]

\begin{definition}\cite{ambrosio2001connected}
An open set $U$ is called decomposable if there exist open sets $U_1, U_2$ such that $U = U_1\cup U_2$ (in the $L^1$-sense) and $P(U) = P(U_1) + P(U_2)$. It is called indecomposable if it is not decomposable.
\end{definition}

Heuristically, a set is decomposable if we need not create significant new boundaries when cutting it into pieces. It was shown recently   \cite{novaga-etal} that for (essentially) bounded sets $E\subset\R^2$ such that $\partial E = \partial_*E$ modulo sets of zero $\H^1$-measure the identity 
\[
\overline{P_C}(E) = \overline {P_C^r}(E) = P(E) + 2\,St(E)
\]
holds where $St(E)$ is the length of the Steiner tree of $\overline E$, i.e.\ 
\[
St(E) = \inf\left\{\H^1(K)\:|\: E\cup K \text{ connected}\right\}.
\]
Above, $\H^1$ denotes the $1$-dimensional Hausdorff measure on $\R^2$ and $\partial_*E$ is the essential boundary of $\overline E$, see \cite[Definition 3.60]{ambrosio2000functions}.
For the existence of Steiner trees, their properties and regularity see \cite{paolini2013existence}.

In this article, we develop a phase-field energy functional which approximates the connected relaxation of the perimeter functional in the sense of $\Gamma$-convergence. For technical reasons, we prefer to work on a bounded domain $\Omega$, so we introduce a similar relaxation in this setting:
\[
\overline{P_{C,cc,\Omega}^r}(E) : = \inf\left\{\liminf_{n\to \infty} P(E_n)\:\bigg|\: E_n\stackrel{L^1}\to E, \: E_n\cc\Omega, \:E_n \text{ connected and $C^\infty$-smooth} \right\}.
\]

The notation $E\cc\Omega$ signifies that $E\subset \overline E\subset\Omega$ and that $\overline E$ is compact. In general, forcing sets to remain within $\Omega$ may force us to make longer connections than the $\R^2$-Steiner tree which leads to $\overline{P_{C, cc, \Omega}^r}(E) \neq \overline{P_C}(E)$. We do not provide an explicit characterization of the lsc envelope in this case, nor do we discuss the relationship of other possible relaxations. If $\Omega$ is convex, on the other hand, then a connected set approximating $E$ `gains nothing' by leaving $\Omega$, and $\overline {P _ {C, cc, \Omega}^r} = \overline{P_C}$. We prove a slightly more general statement.

\begin{lemma}\label{lemma convex}
Assume that the convex hull of $E$ is contained in $\Omega$. Then $\overline{P_{C,cc,\Omega}^r} (E) = \overline{P_C^r}(E)$.
\end{lemma} 

Before we prove Lemma \ref{lemma convex}, we introduce a separate useful statement. We assume this to be well-known, but have been unable to find a reference for it.

\begin{lemma}\label{lemma retraction}
Let $U\subset \R^n$ be a convex open set and $K\subset U$ compact. Then there exists a $C^\infty$-diffeomorphism $\phi:\R^n\to U$ such that
\begin{enumerate}
\item $\phi(x) = x$ for all $x\in K$ and
\item $|\phi(x) - \phi(y)|\leq |x-y|$ for all $x, y\in \R^n$.
\end{enumerate}
\end{lemma}

Since the proof of Lemma \ref{lemma retraction} is unrelated to the main points of this article, we have moved it to the appendix. Now we can prove Lemma \ref{lemma convex}.

\begin{proof}[Proof of Lemma \ref{lemma convex}]
It is clear that 
\[
\overline{P_{C,cc,\Omega}^r} (E) \geq \overline{P_C^r}(E)
\]
since fewer sets are admissible in the approximation process, so it suffices to prove the inverse inequality. Without loss of generality, we may also assume that $E\neq \emptyset$ and $P(E)<\infty$.

{\bf Step 1.} Suppose that $0\in \conv(E)$ and denote $E_n:= \frac{n-1}n E$. Since $\conv(E)$ is open, there exists $r>0$ such that $B_r(0)\in \conv(E)$, thus if $x\in E_n$ and $y\in\R^2$ with $|y-x|< \frac rn$, then $\frac{n}{n-1}x\in E$ and thus
\[
 y = x+(y-x) =  \left(1- \frac1n\right)\left[\frac n{n-1} x\right] + \frac1n \left[n\cdot (y-x)\right] \in \conv(\conv(E)) = \conv(E)
\]
which shows that $E_n\cc\conv(E)\subset\Omega$. Furthermore, $P(E_n) = \frac{n-1}n P(E) \leq P(E)$, so by $BV$-compactness there exists a set $E_\infty\subset \conv(E)$ such that $E_n\to E_\infty$ in $L^1$ (up to a subsequence).
Now if $x\in E$, there exists $r>0$ such that $B_r(x)\subset E$ and thus in particular there exists $N\in \N$ such that 
\[
\frac{n}{n-1}\, y \in B_r(x) \subset E \qquad\forall\ y\in B_{r/2}(x), n\geq N
\]
since $f_n(x) = \frac{n}{n-1}x$ converges to the identity map locally uniformly on $\R^n$. Hence $B_{r/2}(x)\subset E_n$ for all $n\geq N$ and therefore also $B_{r/2}(x)\subset E_\infty$. In total, this implies that $E\subset E_\infty$, and since $|E_n| = \frac{n-1}{n}\,|E|\leq |E|$, it follows that $|E_\infty|\leq |E|$ which combines to the statement that $E_\infty=E$.
The uniqueness of the limit shows that in fact $E_n\to E$ also without choosing a subsequence.

Assuming for the moment that $\overline{P^r_C}(E_n) \equiv \overline{P^r_{C,cc,\Omega}}(E_n)$ for all $n\in \N$, we take a sequence of smooth connected sets $E_n^k\cc\Omega$ such that
\[
\lim_{k\to \infty} E_n^k = E_n, \qquad \lim_{k\to \infty} P(E_n^k) = \overline{P_{C,cc,\Omega}^r}(E_n) = \overline{P^r_C}(E_n)
\]
and a diagonal sequence $E_n' = E_n^{k_n}$ such that $E_n'\to E$ and
\[
 \lim_{n\to \infty} P(E_n') = \lim_{n\to \infty}\lim_{k\to\infty} P(E_n^k) = \lim_{n\to\infty} \overline{P_C^r}(E_n) = \lim_{n\to \infty}\frac{n-1}n\, \overline{P_C^r}(E) = \overline{P_C^r}(E)
\]
whence $\overline{P^r_{C,cc,\Omega}}(E) \leq \lim_{n\to \infty} P(E_n') = \overline{P^r_C}(E)$ and thus $\overline{P^r_{C,cc,\Omega}}(E) =  \overline{P^r_C} (E)$.

{\bf Step 2.} It remains to prove that $\overline{P_{C,cc,\Omega}^r}(E_n) = \overline{P^r_C}(E_n)$ holds for all $n\in\N$. Fix any $n\in \N$ and let $E_n^k$ be a sequence of smooth connected subsets of $\R^2$ such that 
\[
\lim_{k\to \infty} E_n^k = E_n, \qquad \lim_{k\to \infty} P(E_n^k) =  \overline{P^r_C}(E_n).
\]
Now let $\phi$ be a diffeomorphism as in Lemma \ref{lemma retraction} with $U = \conv(E)$ and $K = \overline{E_n}$. Since $\phi$ is a $C^\infty$-diffeomorphism, $\phi(E_n^k)$ is also a connected open set with $C^\infty$-boundary, but additionally $\phi(E_n^k)\cc \conv(E)$. Since $\phi$ is $1$-Lipschitz, we find that
\[
P(\phi(E_n^k)) = \H^1\left(\partial \phi(E_n^k)\right) = \H^1\left(\phi(\partial E_n^k)\right) \leq \H^1\left(\partial E_n^k\right) = P(E_n^k)
\]
and also
\[
\big| \phi(E_n^k) \Delta E_n\big| = \big|\phi(E_n^k)\Delta \phi(E_n)\big| = \big|\phi(E_n^k\Delta E_b)\big| \leq \big|E_n^k\Delta E_n\big|
\]
since $\phi = \id$ on $E_n$. Thus $\phi(E_n^k)\to E_n$ in $L^1$ and
\[
\overline{P^r_C}(E_n) = \lim_{k\to \infty} P(E_n^k) \geq 
\liminf_{k\to\infty} P(\phi(E_n^k))\geq \overline{P^r_{C,cc,\Omega}}(E_n) 
\]
by the definition of the lower semi-continuous envelope. Since the opposite inequality is obvious, we find that $\overline{P^r_C}(E_n) = \overline{P^r_{C,cc,\Omega}}(E_n)$ which concludes the proof.
\end{proof}

This result remains true if we consider the relaxation of the perimeter functional under approximation by {\em simply connected} sets. Also for this functional, an explicit characterization is available due to \cite{novaga-etal} as
\[
\overline{P_{sc}}(E) = \overline{P_{sc}^r}(E) = P(E) + 2\,St(E) + 2\,St(E^c)
\]
with notation analogous to the connected relaxation. We will briefly come back to this problem in Theorem \ref{theorem simply connected}.

\begin{remark}
Let $E \subset \R^3$ be the set given by 
\[
E = B_1(-Re_1) \cup B_1(Re_1)
\]
for some $R\gg 1$. The sets
\[
E_n = E \cup \left\{ (x,y,z) \in \R^3\:|\: -R < x<R, \:y^2 + z^2 < \frac1{n^2}\right\}
\]
are connected, open, have a Lipschitz boundary and satisfy $P(E_n) \to P(E)$. A slightly modified sequence of sets can be constructed to have $C^\infty$-boundaries. Thus, since we can connect two components of an open set with a tube of small volume and perimeter, we do not expect the relaxation of the perimeter functional under a connectedness constraint to exhibit any interesting behaviour in ambient spaces of dimension $\geq 2$. 
\end{remark}

\subsection{The phase-field model}\label{section phase field functional}

We choose the classical Modica-Mortola approximation \cite{modica:1987us}
\[
\F_\eps(u) = \frac1{c_0}\int_\Omega\frac\eps2\,|\nabla u|^2 + \frac1\eps\,W(u)\dx 
\]
of the perimeter functional where $W(u) = u^2(u-1)^2$ is a double-well potential and $c_0$ is a normalising constant given by
\[
c_0 = \int_0^1 \sqrt{2\,W(s)}\ds.
\]
To incorporate a connectedness constraint, we follow an idea developed by two of the authors for a problem of surfaces in a three-dimensional ambient space \cite{MR3590663} based on a similar model for the two-dimensional Steiner problem and related questions \cite{bonnivard:2014tw}. Due to its novelty, we include a heuristic motivation here.

Recall that an open set $E$ in $\R^n$ is connected if and only if it is path-connected, so if and only if for every $x,y\in \R^n$ there exists a continuous curve $\gamma:[0,1]\to \R^n$ such that $\gamma(0) = x$, $\gamma(1)=y$ and $\gamma(t)\in E$ for all $t\in [0,1]$. It is well-known that we may assume that $\gamma$ is smooth. 

We introduce a quantitative notion of path-connectedness to generalize this concept. Let $F\in C^0(\R^n)$ be a function such that
\[
F(x) = 0 \text{ if }x\in \overline{E}, \qquad F(x)>0\text{ if }x\in \R^n\setminus\overline E
\]
and define the geodesic distance
\[
d^F(x,y) = \inf\left\{\int_\gamma F\d\H^1\:\bigg|\:\gamma(0) = x, \:\gamma(1) = y,\: \gamma:[0,1]\to\R^n \text{ Lipschitz}\right\}.
\]
Then, if $x, y$ lie in the same connected component of $E$ (or rather, $\overline E$), then $d^F(x,y)=0$, while if they lie in different connected components, we would expect them to be separated by a positive $d^F$-distance (at least in `nice' cases). So we think of $d^F(x,y)$ as a quantitative measure of the path-disconnectedness of the set $E$ at the points $x,y\in E$. To obtain a single number to measure the total path-disconnectedness of $E$, we can consider a double-integral
\[
\int_E\int_E d^F(x,y)\dx\dy\quad\text{or}\quad \int_{\R^n}\int_{\R^n} \beta(x)\,\beta(y)\,d^F(x,y)\dx\dy
\]
where $\beta:\R^n\to\R$ is a measurable function such that $\beta(x) = 0$ if $x\in\R^n\setminus E$ and $\beta(x)>0$ if $x\in E$.

This can be adapted to a phase-field setting as follows. We want to approximate a set $E$ by connected open sets $E_n$. Letting $u$ denote the phase-field function approximating the characteristic function of $E$ for some phase-field parameter $\eps$, this corresponds to keeping the set $\{u\approx 1\}$ connected. More precisely, we choose $0<s<1/2$ and penalize the quantitative total disconnectedness of the set $\{1-\eps^s < u_\eps\}$. So take Lipschitz-functions $\beta_\eps, F_\eps$ which are monotone increasing/decreasing respectively such that
\[
\beta_\eps(z) = \begin{cases} 0 & z\leq 1- 2\eps^s\\ 1 &z\geq 1-\eps^s\end{cases}, \qquad F_\eps(z) = \begin{cases} 1 & z\leq 1- 2\eps^s\\ 0 &z\geq 1-\eps^s\end{cases}.
\]

Note that if $\gamma$ is a Lipschitz-curve, we can take the trace of a $W^{1,2}$-function $u$ on $\gamma$, so that a geodesic distance with weight $F_\eps(u)$ can be defined in the same way as above, albeit with a weight which is only non-negative, bounded and measurable on the curve. We introduce the `diffuse connectedness functional'
\begin{equation} \label{eq:top_energy}
\C_\eps(u) = \int_\Omega\int_\Omega \beta_\eps(u(x))\:\beta_\eps(u(y))\:d^{F_\eps(u)}(x,y)\dx \dy
\end{equation}
and the total energy of a phase-field
\[
\E_\eps:W^{1,2}_0(\Omega)\to \R, \quad \E_\eps(u) = \F_\eps(u) + \eps^{-\kappa}\,\C_\eps(u)
\]
for some $\kappa>0$ which measures the perimeter and penalizes disconnectedness.

\begin{remark} 
We can allow different double-well potentials $W$, but we need to couple the parameter $s$ in the choice $\beta_\eps, F_\eps$ to the order at which $W$ vanishes at the potential wells.
\end{remark}

\subsection{The sharp interface limit}\label{section main}

In this section we prove our main result, which essentially states that the functionals $\E_\eps$ approximate the relaxed connected perimeter.

\begin{theorem}\label{theorem main}
Let $\Omega\subset \R^2$ be open and bounded. Then
\[
\left[\Gamma(L^1)-\lim_{\eps\to 0} \E_\eps\right] (u) = \begin{cases}\overline{P_{C,cc,\Omega}^r}(\{u=1\}) &\text{if } u\in BV(\overline\Omega, \{0,1\})\\ +\infty &\text{else.}\end{cases}
\]
\end{theorem}

In particular, if $\Omega$ is convex and $P(E) = \H^1(\partial E)$ for $E=\{u=1\}$, the $\Gamma$-limit is known to be 
\[
\overline{P_C} (E)= P(E) + 2\,St(E)
\]
due to Lemma \ref{lemma convex} and \cite[Theorems 3.4 and 3.7]{novaga-etal}.

\begin{proof}[Proof of the $\limsup$-inequality.]
This construction is classical and thus we only sketch the proof. For more detailed arguments concerning the Modica-Mortola functional, see \cite{modica:1987us}. Let $u\in BV(\overline\Omega, \{0,1\})$ and denote $E:=\{u=1\}$. We want to construct a sequence of phase-fields $u_\eps$ such that 
\[
\lim_{\eps\to 0}\E_\eps(u_\eps) = \overline{P^r_{C,cc,\Omega}}(E).
\]
Take a sequence of connected sets $E_n$ such that 
\[
E_n\cc\Omega, \qquad \partial E_n\in C^\infty, \qquad
E_n \stackrel{L^1}\longrightarrow E, \qquad P(E_n) \longrightarrow \overline{P_{C,cc,\Omega}^r}(E).
\]
For every $n\in \N$, we may pick $r_n$ such that the tubular neighbourhood
\[
U_n = \{x\in \R^n\:|\: \dist(x, \partial E_n) < r_n\}
\]
is diffeomorphic to $\partial E_n \times (-r_n, r_n)$ via the map
\[
\Phi:\partial E_n \times (-r_n, r_n) \to U_n, \qquad \Phi(x, t) = x + t\,\nu_x.
\]
Without loss of generality, we assume that the sequence $r_n$ is strictly monotone decreasing to zero. For $r_{n+1}^2 \leq \eps < r_n^2$, we insert the usual recovery sequence for $E_n$,
\[
u_\eps(x) = q\left(\frac{\sdist(x,\partial E_n)}\eps\right) \cdot\chi_\eps\,,
\]
where $q$ solves the $1$-dimensional cell problem
\[
q'' - W'(q) = 0, \qquad q(-\infty) = 0, \qquad q(0)= \frac12, \qquad  q(+\infty) = 1,
\]
$\sdist(x, \partial E_n) = \dist(x, E_n^c) - \dist(x, E_n)$ is the signed distance function from $\partial E_n$ chosen to be positive inside $E_n$ and $\chi_\eps$ is a cut-off function to ensure that $u_\eps=0$ close to $\partial\Omega$. Then also the set $\{u_\eps>1-\eps^s\}$ is diffeomorphic to $E_n$, thus connected, and $\C_\eps(u_\eps)\equiv 0$. It is well-known that
\[
\F_\eps(u_\eps) - P(E_n) \to 0
\]
as $\eps \to 0$ (where $n=n_\eps$ is the corresponding index),
so in total we have shown that $\E_\eps(u_\eps) \to \overline{P_{C,\Omega}} (E)$ as required.  
\end{proof}

\begin{proof}[Proof of the $\liminf$-inequality.]
{\bf Preliminaries and heuristics.}
Let $u_\eps\in W_0^{1,2}(\Omega)$ be a sequence of functions such that $u_\eps\to u$ in $L^1(\Omega)$. Without loss of generality we may assume that $\liminf_{\eps\to 0} \E_\eps(u_\eps) < \infty$. As for the Modica-Mortola functional, this implies that $u$ is the characteristic function of a set of finite perimeter in $\overline\Omega$, so 
we need to show that 
\[
\liminf_{\eps\to 0} \E_\eps(u_\eps) \geq \overline{P_{C,cc,\Omega}^r}(\{u=1\}).
\]
We denote $E:= \{u=1\}$.
Since the energy $\E_\eps(u_\eps)$ decreases when we truncate $u_\eps$ from above at $1$ and from below at $0$, we may assume that $u_\eps \in W^{1,2}_0(\Omega, [0,1])$, and by the density of smooth functions even that $u_\eps \in C^\infty_c(\Omega, [0,1))$. 

Let $0<\delta < \frac{c_0}2$ and consider the primitive function 
$
G(z) = \int_0^z\sqrt{2\,W(s)\,}\ds.
$
Using the co-area formula for $BV$-functions we obtain that
\begin{align*}
\int_\Omega \frac\eps2\,|\nabla u_\eps|^2 + \frac{W(u_\eps)}{\eps}\dx &\geq \int_\Omega \sqrt{2 W(u_\eps)}|\nabla u_\eps|\dx\\
	&= \int_\Omega |\nabla G(u_\eps)|\dx\\
	&= \int_{0}^{c_0} \int_\Omega \d|\nabla \chi_{\{G(u_\eps)>t\}}| \dt
\end{align*}
for $\chi_A$ the characteristic function of a set $A$, so there exists $t_\eps \in (\delta, c_0-\delta)$ such that 
\[
\int_\Omega \d |\nabla \chi_{\{G(u_\eps)>t_\eps\}}| \leq \frac{1}{c_0-2\delta}\int_\Omega \frac\eps2|\nabla u_\eps|^2+ \frac{W(u_\eps)}\eps\dx.
\]
Since almost all $t$ are regular values of $G(u_\eps)$, we can even pick $t_\eps$ such that $\{G(u_\eps)=t_\eps\}$ is a $C^\infty$-submanifold of $\R^2$. Since the level set is additionally closed and bounded, we see that $\{G(u_\eps)=t_\eps\}$ is a compact manifold -- in particular, $\{G(u_\eps)=t_\eps\}$ has only a finite number $M_\eps$ of connected components. Since $\nabla G(u_\eps)$ does not vanish on $\{G(u_\eps)=t_\eps\}$ by assumption and since $u_\eps =0 < t_\eps$ on $\partial\Omega$, we see that
\[
\{G(u_\eps)=t_\eps\} = \partial\{G(u_\eps)>t_\eps\}.
\]
 
Now, let us go through the heuristic of the proof: If we could show that $\{G(u_\eps)>t_\eps\}$ were connected, we would be done, arguing that 
\[
\{G(u_\eps)>t_\eps\} = \{u_\eps > G^{-1}(t_\eps)\} \stackrel{L^1}\longrightarrow \{u=1\}
\]
and then concluding that 
\begin{align*}
\overline{P^r_{C,cc,\Omega}}(E) &\leq \liminf_{\eps\to 0} P\big(\{G(u_\eps)>t_\eps\}\big)\\
	&\leq \liminf_{\eps\to 0} \frac{1}{c_0-2\delta}\int_\Omega \frac\eps2|\nabla u_\eps|^2+ \frac{W(u_\eps)}\eps\dx\\
	&\leq \frac{c_0}{c_0-2\delta} \liminf_{\eps\to 0} \E_\eps(u_\eps)
\end{align*}
for all $0<\delta<\frac{c_0}2$. Then, taking $\delta\to0$, we would obtain the $\Gamma-\liminf$-inequality
\[
\overline{P^r_{C,cc,\Omega}}(E)\leq \liminf_{\eps\to 0} \E_\eps(u_\eps).
\]
In general, there is no good reason for $\{G(u_\eps)>t_\eps\}$ to be connected since the super-level set is highly sensitive to very slight perturbations which are barely visible in the Modica-Mortola energy -- however, the energy contribution of $\C_\eps$ prevents the set $\{u_\eps>1-\eps^s\}$ from being `too disconnected', so we can take a slight modification of the set which barely changes area or perimeter, but makes it connected.

{\bf Step 1.} In this step, we show that 
\[
\lim_{\eps\to 0} E_\eps = E
\]
in the $L^1$-topology of open sets for all sets $E_\eps$ such that
\[
\{u_\eps>1-\eps^s\} \subset E_\eps \subset \{u_\eps>\eps^s\}.
\]
In particular we note that
\[
\{u_\eps > 1-\eps^s \} \subset \{u_\eps> G^{-1}(t_\eps)\} \subset \{u_\eps>\eps^s\}
\]
since $\delta< t_\eps<c_0-\delta$, so $G^{-1}(t_\eps)$ is bounded away from $0$ and $1$. Second, we note that 
\[
|\{\eps^s < u_\eps < 1-\eps^s\}| \leq \frac{\eps}{W(\eps^s)}\int_\Omega \frac{W(u_\eps)}\eps\dx \leq \frac{C \eps}{\eps^{2s}} = C\,\eps^{1-2s}\to 0,
\]
where $C$ is an energy bound uniform in $\eps$, so that any set containing $\{u_\eps>1-\eps^s\}$ and contained in $\{u_\eps>\eps^s\}$ has the same $L^1$-limit (if one of them exists). Here we use that $W$ vanishes quadratically at the potential wells, for other double-well potentials, other $s$ may be admissible. Now observe that
\begin{align*}
\| u_\eps - \chi_{\{u_\eps>1-\eps^s\}}\|_{L^1} &= \int_{\{u_\eps> 1-\eps^s\}} |u_\eps - 1|\dx + \int_{\{\eps^s < u_\eps < 1-\eps^s\}}|u_\eps|\dx + \int_{\{u_\eps<\eps^s\}}|u_\eps|\dx\\
	&\leq \eps^s\cdot|\Omega| + C\,\eps^{1-2s} + \eps^s\cdot |\Omega| \to 0
\end{align*}
so that $u_\eps$ and $\chi_{\{u_\eps > 1-\eps^s\}}$ have the same $L^1$-limit $u = \chi_E$, in other words
\[
\lim_{\eps\to 0} \{u_\eps > 1-\eps^s\} = E
\]
in the $L^1$-topology.

{\bf Step 2.} In this step, we eliminate the connected components of the approximating set which we deem too small to matter.
Denote the connected components of $\{G(u_\eps)>t_\eps\}$ by $U_{i,\eps}$, $i = 1, \dots, M_\eps$, where the components are ordered by volume:
\[
|U_{1,\eps}| \:\geq\: \dots\:\geq |U_{M_\eps,\eps}|.
\]
Denote
\[
E_\eps = \{G(u_\eps)>t_\eps\} = \bigcup_{i=1}^{M_\eps} U_{i, \eps}.
\]
 Applying the iso-perimetric inequality, we observe that for $N<M_\eps$ we have
\begin{align*}
\left|\bigcup_{i=N}^{M_\eps} U_{i,\eps}\right| &= \sum_{i=N}^{M_\eps} \left|U_{i,\eps}\right|&
	&\leq \sqrt{|U_{N,\eps}|\,}\sum_{i=N}^{M_\eps} \sqrt{|U_{i,\eps}|\,}&
	&\leq \sqrt{|U_{N,\eps}|\,}\sum_{i=N}^{M_\eps} \sqrt{\frac1{4\pi} \Per( U_{i,\eps})^2}\\
	&= \sqrt{\frac {|U_{N,\eps}|}{4\pi}} \sum_{i=N}^{M_\eps} \Per( U_{i,\eps})&
	&= \sqrt{\frac {|U_{N,\eps}| }{4\pi}}\:\Per \left(\bigcup_{i=N}^{M_\eps} U_{i,\eps}\right)&
	&\leq \sqrt{|U_{N,\eps}|}\: \frac{\Per(E_\eps)}{\sqrt{4\pi}},
\end{align*}
so if $U_{N,\eps}$ carries little mass, then all the remaining components together have little mass as well.
The identity
\[
\sum_{i=N}^{M_\eps} \Per( U_{i,\eps}) = \Per \left(\bigcup_{i=N}^{M_\eps} U_{i,\eps}\right)
\]
holds easily since $E_\eps$ is a smooth set whose boundary has finitely many connected components.
Choose $N_\eps\leq M_\eps$ such that 
\[
\left| U_{i, \eps}\right| \geq \frac1{|\log\eps|}\qquad\forall\ i= 1, \dots, N_\eps, \qquad |U_{N_\eps+1, \eps}| < \frac1{|\log\eps|}.
\]
It may happen that $N_\eps = M_\eps$ -- this will rather simplify the proof, so we do not consider that case. We take 
\[
E_\eps' = \bigcup_{i=1}^{N_\eps} U_{i,\eps}
\]
and note that $P(E_\eps') \leq P(E_\eps)$ (since we only remove boundary components) and still $E_\eps' \to E$ in $L^1$. 

\begin{figure}[ht!]
\begin{center}
\includegraphics[]{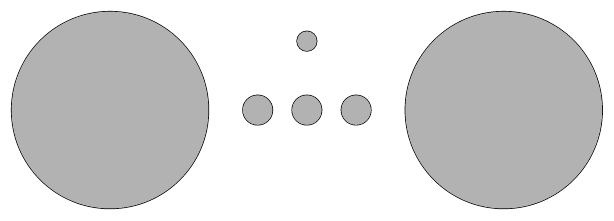}\\
\includegraphics[]{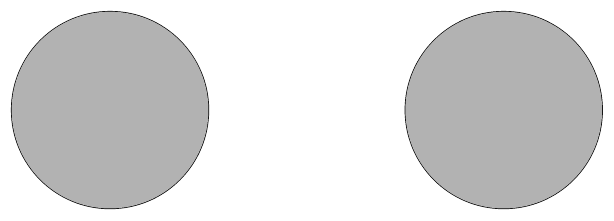}\\
\includegraphics[]{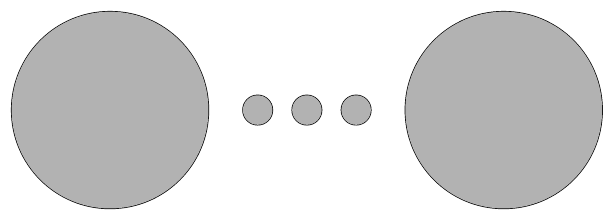}\\
\includegraphics[]{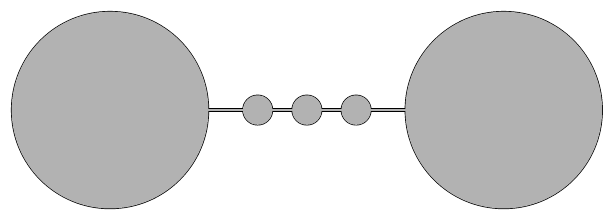}\\
\includegraphics[]{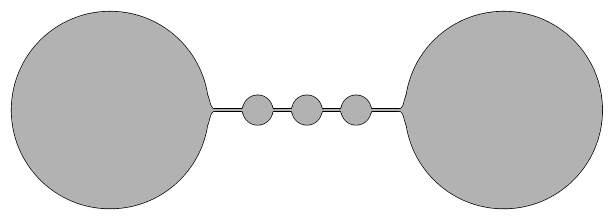}
\end{center}
\caption{Schematic illustration of the modification process. From top to bottom $E_\eps$ (all components), $E_\eps'$ (only the large components), $E_\eps''$ (also the small components on connecting curves), $E_\eps'''$ (same components, connected with tubes), $E_\eps''''$ (smoothed out version of $E_\eps'''$).}
\end{figure}

{\bf Step 3.} In this step, we show that it is possible to connect all the components $U_{i,\eps}$ of $E_\eps'$ without changing the $L^1$-limit or increasing the perimeter by much. We note that the number of large connected components cannot increase too quickly in $\eps$ since
\[
\frac{N_\eps}{|\log\eps|} \leq N_\eps \,|U_{N_\eps,\eps}| \leq \sum_{i=1}^{N_\eps}|U_{i,\eps}| = |E_\eps'| \leq |\Omega|\qquad\Ra\qquad N_\eps \leq |\Omega|\,|\log\eps|.
\]
Furthermore, we note that for $i\leq N_\eps$
\[
\big|U_{i,\eps} \cap \{u_\eps>1-\eps^s\}\big| \geq |U_{i,\eps}| - \big|\{x\in \Omega\:|\:\eps^s < u_\eps < 1-\eps^s\}\big| \geq |\log\eps|^{-1} - C\,\eps^{1-2s} \geq \frac1{2\,|\log\eps|}
\]
for all small enough $\eps>0$, so that within each connected component $U_{i,\eps}$ we have a large volume on which $u_\eps$ is between $1-\eps^s$ and $1$. 
Now let $2\leq i\leq N_\eps$. Then we know that
\begin{align*}
C&\geq \eps^{-\kappa}\int_\Omega\int_\Omega \beta_\eps(u_\eps(x))\:\beta_\eps(u_\eps(y))\:d^{F_\eps(u_\eps)}(x,y)\dx\dy\\
	&\geq \eps^{-\kappa}\int_{U_{i,\eps}\cap \{u_\eps>1-\eps^s\}}\int_{U_{1,\eps}\cap\{u_\eps>1-\eps^s\}} d^{F_\eps(u_\eps)}(x,y)\dx\dy\\
	&=\eps^{-\kappa}\,\big|U_{i,\eps}\cap \{u_\eps>1-\eps^s\}\big|\: \big|U_{1,\eps}\cap \{u_\eps>1-\eps^s\}\big|\:\dist^{F_\eps(u_\eps)}(U_{1,\eps}, U_{i,\eps})\\
	&\geq \frac{1}{4\,|\log\eps|^2 \eps^\kappa}\:\dist^{F_\eps(u_\eps)}(U_{1,\eps}, U_{i,\eps})
\end{align*}
from the energy bound, so 
\[
\dist^{F_\eps(u_\eps)}(U_{1,\eps}, U_{i,\eps}) \leq 4\,C\,|\log\eps|^2\eps^\kappa.
\]
Here we denoted as usual
\[
\dist^{F(u_\eps)}(A, B) = \inf\left\{d^{F(u_\eps)}(x,y)\:|\:x\in A, y\in B\right\}.
\]
This means that there exist two points $x\in U_{1,\eps}$, $y\in U_{i,\eps}$ and a Lipschitz curve $\gamma_i$ from $x$ to $y$ inside $\Omega$ such that
\[
\int_{\gamma_i} F_\eps(u_\eps)\d\H^1 \leq 8\,C\,|\log\eps|^2\eps^\kappa.
\]
Without loss of generality, we may assume that the curve is $C^\infty$-smooth, and we observe that in particular
\[
\H^1 \left(\gamma_i \setminus \{G(u_\eps)>t_\eps\}\right) = \int_{\gamma_i \setminus \{G(u_\eps)>t_\eps\}} F_\eps(u_\eps)\d\H^1 \leq 8\,C\,|\log\eps|^2\eps^\kappa.
\]
Making the curve potentially shorter, we may assume that it has a unique point of entry and point of exit from every connected component of $E_\eps=\{G(u_\eps)>t_\eps\}$ since every connected open subset of $\R^n$ is also path-connected. 
If $\gamma_i$ happens to pass through a connected component of $E_\eps$ which we had eliminated before, we need to add it back:
\[
E_\eps'' = \bigcup_{i=1}^{N_\eps}\left(U_{i,\eps}\cup \bigcup_{U_{j,\eps}\cap \gamma_i\neq\emptyset} U_{j,\eps}\right)
\]
and note that still $E_\eps'' \to E$ and $P(E_\eps'') \leq P(E_\eps)$. 
Now we know that 
\[
E_\eps'' \cup \bigcup_{i=1}^{N_\eps}\left(\gamma_i\setminus E_\eps''\right)
\]
is connected (but not open). Since $\gamma_i$ is $C^\infty$-smooth and only meets finitely many connected components, there exists $\rho_\eps>0$ such that the tubular neighbourhood 
\[
B_{\rho_\eps}(\gamma_i) = \{ x\in \R^2\:|\: \dist (x, \gamma_i)< \rho_\eps\}
\]
is compactly contained in $\Omega$, such that
\[
\H^1\left( \big(\partial B_{\rho_\eps}(\gamma_i)\big) \setminus E_\eps''\right) \leq 3\, \H^1(\gamma_i\setminus E_\eps'')
\]
and such that the tubes only add a negligible amount of area. Now, using that we only had at most $O(|\log\eps|)$ tubes to add which were all small compared to $\frac1{|\log\eps|}$, we observe that
\[
E_\eps''' = E_\eps'' \cup \bigcup_{i=1}^{N_\eps} B_\eps(\gamma_i)
\]
is open, connected, converges to $E$ as $\eps\to 0$ and satisfies 
\[
\liminf_{\eps\to 0} \left(\frac{c_0}{c_0-2\delta} \,\E_\eps(u_\eps) - P(E_\eps''')\right) \geq 0.
\]
We also note that $E_\eps'''$ has a smooth boundary (if we choose $\rho_\eps$ small enough) except at the finitely many points where the tubular neighbourhoods hit the connected components. We can smooth those corners out locally to a set $E_\eps''''$ with the exact same properties otherwise, but a boundary which is actually $C^\infty$-smooth. This proves the theorem.
\end{proof}

\subsection{Extensions and further observations}\label{section extensions}

As for the pure Modica-Mortola functional, a compactness result holds.

\begin{remark}
If $u_\eps \in W^{1,2}_0(\Omega)$ is a sequence such that $\liminf_{\eps\to 0}\E_\eps(u_\eps)<\infty$, then there exists a subsequence $\eps\to 0$ and a set of finite perimeter $E\subset \Omega$ such that
\[
u_\eps \Lto \chi_E, \qquad P(E) \leq \overline{P_C(E)} \leq \overline{P_{C,cc,\Omega}^r}(E) \leq \liminf_{\eps\to 0} \E_\eps(u_\eps).
\]
In fact, the convergence holds in $L^p(\Omega)$ for all $p<\infty$.
\end{remark}

Let us quickly collect a few thoughts on how similar ideas may be used in related problems. If we define the energy functional $\E_\eps'$ formally given by the same formula on $W^{1,2}(\Omega)$ instead of $W^{1,2}_0(\Omega)$, i.e.\
\[
\E_\eps':W^{1,2}(\Omega)\to \R, \quad \E_\eps'(u) =\frac1{c_0}\int_\Omega\frac\eps2\,|\nabla u|^2 + \frac{W(u)}\eps\dx + \eps^{-\kappa}\,\C_\eps(u)
\]
for some $\kappa>0$, we get a connected relaxation of the relative perimeter.

\begin{theorem}\label{theorem relative perimeter}
Assume that $\Omega\subset \R^2$ is a bounded Lipschitz domain. Then
\[
\left[\Gamma(L^1)-\lim_{\eps\to 0} \E_\eps'\right] (u) = \begin{cases}\overline{P_{C,rel,\Omega}^r}(\{u=1\}) &\text{if } u\in BV(\overline\Omega, \{0,1\})\\ +\infty &\text{else}\end{cases}
\]
where
\[
\overline{P_{C,rel,\Omega}^r}(E) = \inf\left\{\liminf_{n\to\infty} P_\Omega(E_n) \:\bigg|\:E_n\subseteq \Omega, \quad (\partial E_n)\cap \Omega\in C^\infty, \quad E_n\text{ connected}, \quad E_n\Lto E\right\}.
\]
\end{theorem}

The relative perimeter $P_\Omega(E)$ of $E\subseteq \Omega$ is defined as $P_\Omega(E) = |D\chi_E|(\Omega)$ while the full perimeter of a set $E\subset \Omega$ is $P(E) = P_{\R^2}(E) = |D\chi_E|(\overline\Omega) = |D\chi_E|(\R^2)$, i.e.\ the relative perimeter does not count the part of the boundary of $E$ that lies inside the boundary of $\Omega$, see e.g.\ \cite[Definition 3.35]{ambrosio2000functions}.

The proof of the $\liminf$-inequality is the same as that of Theorem \ref{theorem main}, using the famous result that $H=W$ (i.e.\ that smooth functions on $\R^n$ lie dense in $W^{1,2}(\Omega)$) and the relative iso-perimetric inequality which holds in all domains where the Sobolev inequality holds (in particular, Lipschitz domains). The construction of the recovery sequence goes through as before, assuming that the boundary of $\Omega$ is not too wild.

Note that a bounded set $E\subset\R^2$ is simply connected if and only if both $E$ and $\R^2\setminus E$ are connected. This leads us to investigate a modified functional 
\[
\E_\eps'':W^{1,2}_0(\Omega)\to \R, \quad \E_\eps''(u) = \F_\eps(u) + \eps^{-\kappa}\big[\C_\eps^{(1)}(u) + \C_\eps^{(2)}(u)\big]
\]
where
\begin{align*}
\C_\eps^{(1)}(u) &= \int_\Omega\int_\Omega \beta_\eps\big(u(x)\big)\:\beta_\eps\big(u(y)\big)\:d^{F_\eps(u)}(x,y)\dx \dy\\
\C_\eps^{(2)}(u) &= \int_\Omega\int_\Omega \beta_\eps\big(1- u(x)\big)\:\beta_\eps\big(1- u(y)\big)\:d^{F_\eps(1-u)}(x,y)\dx \dy
\end{align*}
i.e.\ $\C_\eps^{(1)}$ as before serves to keep the phase $\{u \approx 1\}$ approximately connected whereas $\C_\eps^{(2)}$ keeps the phase $\{1-u \approx 1\} = \{u\approx 0\}$ connected. We have the following result.

\begin{theorem}\label{theorem simply connected}
\[
\left[\Gamma(L^1)-\lim_{\eps\to 0} \E_\eps''\right] (u) = \begin{cases}\overline{P_{sc,cc,\Omega}^r}(\{u=1\}) &\text{if } u\in BV(\overline\Omega, \{0,1\})\\ +\infty &\text{else}\end{cases}
\]
where
\[
\overline{P_{sc,cc,\Omega}^r}(E) = \inf\left\{\liminf_{n\to\infty} P_\Omega(E_n) \:|\:E_n\subset \Omega, \quad E_n\in C^\infty, \quad E_n \text{ simply connected}, \quad E_n\Lto E\right\}.
\]
\end{theorem}

\begin{proof}
The proof of the $\limsup$-inequality proceeds in the usual way, so we will only look at the necessary modifications for the $\liminf$-inequality. The boundary of the approximating set $E_\eps = \{G(u_\eps)>t_\eps\}$ is a compact embedded $C^\infty$-submanifold of $\R^2$, so the union of finitely many circles which do not touch each other. In particular, if $U_{i, \eps}$ is a connected component of $E_\eps$, it is only in contact with {\em one} connected component of $\overline {E_\eps}^c$. 

This means that if we add a $C^\infty$-curve $\gamma_{i,j,\eps}$ to $E_\eps$ which connects two connected components $U_{i,\eps}$ and $U_{j,\eps}$ in such a way that it has one entry and one exit point to $\overline{E_\eps}$ and no loops, then every connected component of $\overline {E_\eps}^c$ will still be connected after the modification. The same is true after slightly fattening the curve. A simple proof of this fact can be constructed using path-connectedness and the regularity of the approximating sets to look at tubular neighbourhoods of the boundaries and the connecting curve.

Thus we may carefully construct connecting curves $\gamma_{i,j,\eps}$ between components which have no loops and connect components in such a way that there always is only one entry and one exit point for a component, also proceeding iteratively and merging components to be the same after they have been connected before connecting the next one.

After constructing $E_\eps''''$ in such a way, we can modify the complement $\overline{E_\eps''''}^c$ in the same way to make it connected without changing the fact that $E_\eps''''$ is connected, creating a new set $\widehat E_\eps$. By construction, we again barely changed the perimeter and know that both $\widehat E_\eps$ and the complement of its closure are connected. Since $\widehat E_\eps$ is also $C^\infty$-smooth, it follows that also $\widehat E_\eps^c$ is connected, which means that $\widehat E_\eps$ is simply connected. 

As before, this concludes the proof.
\end{proof}

Of course it is possible to combine the previous two extensions. We conclude this section with two notes on possible applications of our approximation results.

\begin{remark}
In order to make use of our functional for image segmentation applications, it is of course possible to add a fidelity term of the form
\[
\F_\mathrm{fid}(u) = \int_{\Omega} \frac{1}{2}\Phi(x)|u(x)-g(x)|^2\,\dx
\]
for a given image $g\colon \Omega \to [0,1]$ and local fidelity prefactor $\Phi\colon \Omega \to [0,1]$. This term simply carries over to the $\Gamma$-limit proved above.
\end{remark}

\section{Numerical Results}\label{section numerics}
As in~\cite{Dondl:2018wb}, we consider a fully discrete gradient flow of the functional
\[
\E_\mathrm{im}(u) = \F_\eps(u) + \eta\,\C_\eps(u) + \delta\F_\mathrm{fid}(u).
\]
The two wells of the function $W$ in $\F_\eps$ are at $0$ and $1$. We only consider a fixed $\eps=5\cdot 10^{-3}$ for these numerical experiments and the functions $\beta_\eps$ and $F_\eps$ used to define $\C_\eps$ are given by
\begin{align*}
\beta_\eps(s) &= \begin{cases}
0 & s\le 1-\alpha \\
\frac{c_1}{2}(s-1+\alpha)^2  & s>1-\alpha
\end{cases} \\
\quad\text{and}\\
F_\eps(s) &= \begin{cases} 
\frac{1}{2}(s-1+\alpha)^2\cdot  & s<1-\alpha\\
0 & s \le 1-\alpha,
\end{cases}
\end{align*}
respectively, with $c_1$ chosen such that $\int_\alpha^1 \beta_\eps(s) \ds = 1$. The value of $\alpha$ for all numerical examples is $0.35$, the value for $\eta$ is $300$ for all experiments where the topological penalty is turned on. The value for $\delta$ varies somewhat from experiment to experiment.

For the finite element implementation of the discrete gradient flow, we use the algorithm described in detail in~\cite{Dondl:2018wb} and a time-step size of $\tau=5\cdot 10^{-8}$. The basic idea is to first separate the set $\{u>1-\alpha\}$ into connected components and then calculate their distances (and the respective variations, both modulo a mesh-dependent factor), by using Dijkstra's algorithm~\cite{dijkstra1959note}. All computations are done on a unit square made up of approximately $2.3\cdot 10^4$ P1 triangle elements. Some numerical experiments are already presented in~\cite{Dondl:2018wb}, we chose to not repeat those here.

\begin{figure}[ht!]\begin{center}
\raisebox{-0.5\height}{\includegraphics[height=3.7cm]{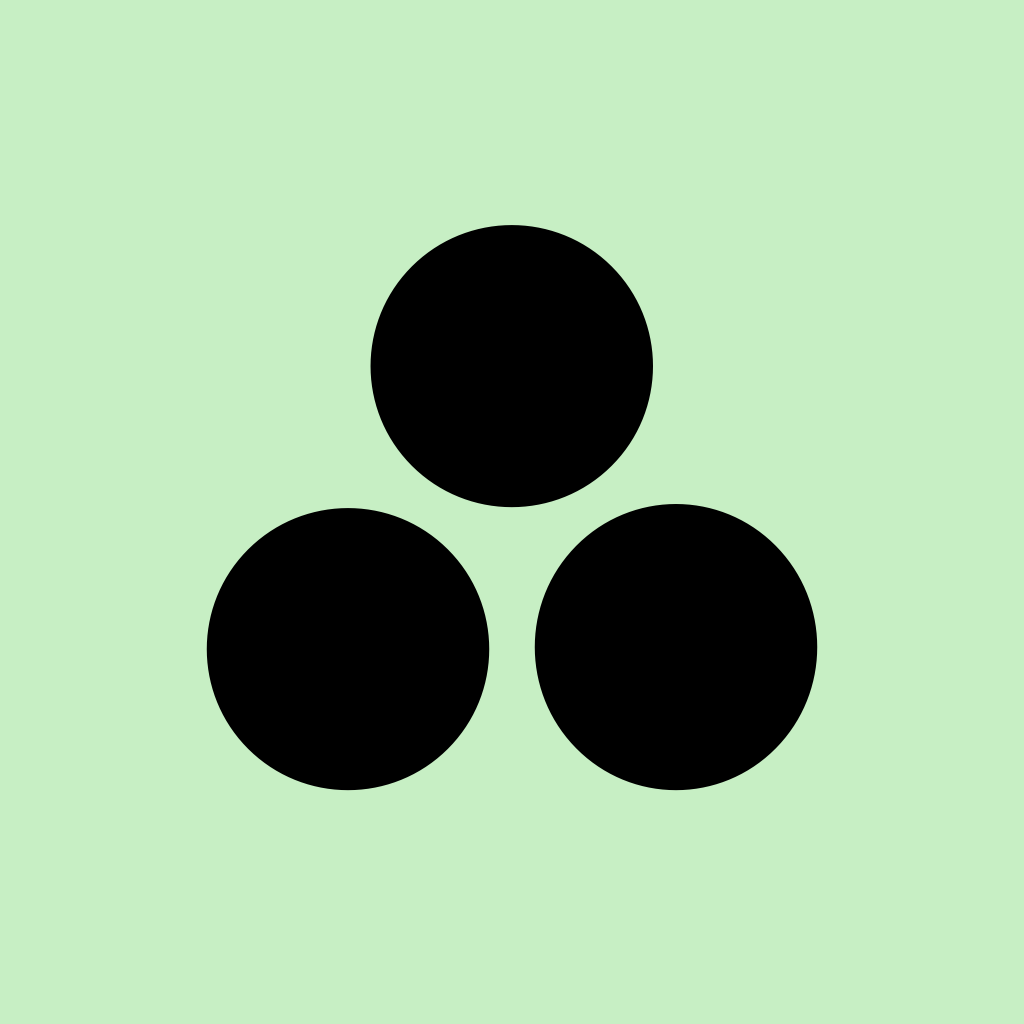}}
\raisebox{-0.5\height}{\includegraphics[height=3.7cm, clip, trim = 36.98cm 18.21cm 36.98cm 18.21cm]{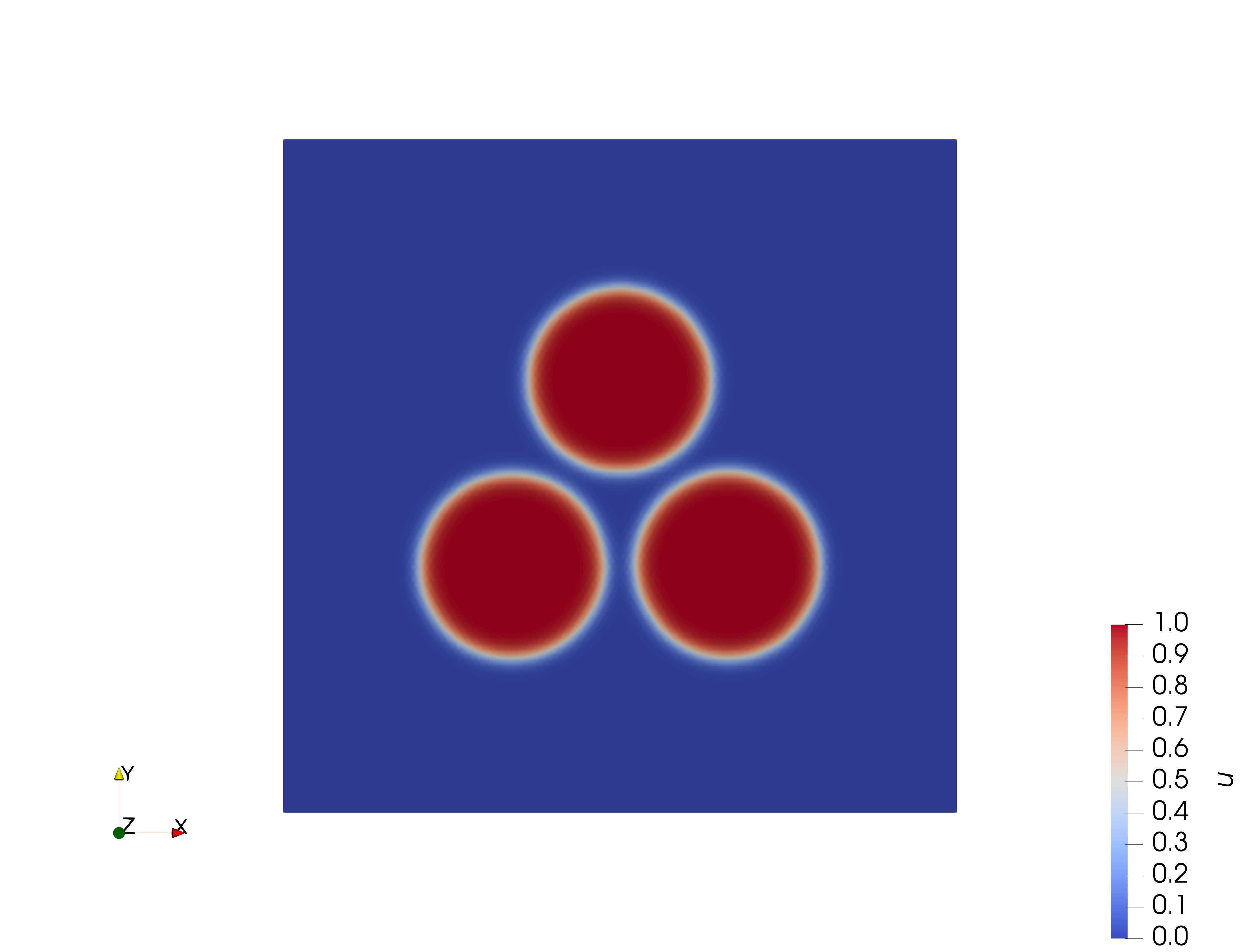}}
\raisebox{-0.5\height}{\includegraphics[height=3.7cm, clip, trim = 36.98cm 18.21cm 36.98cm 18.21cm]{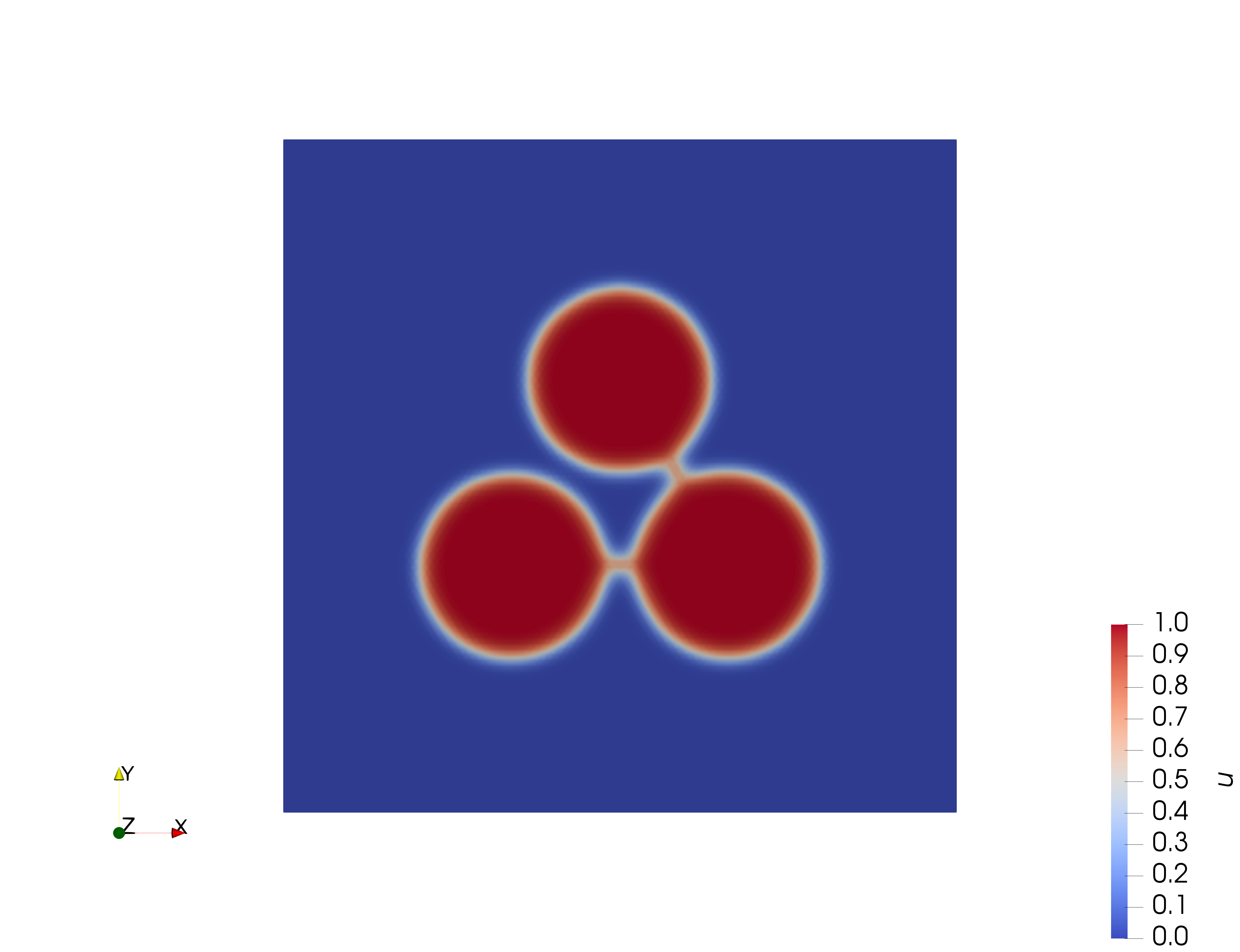}}
\raisebox{-0.5\height}{\includegraphics[height=2.5cm]{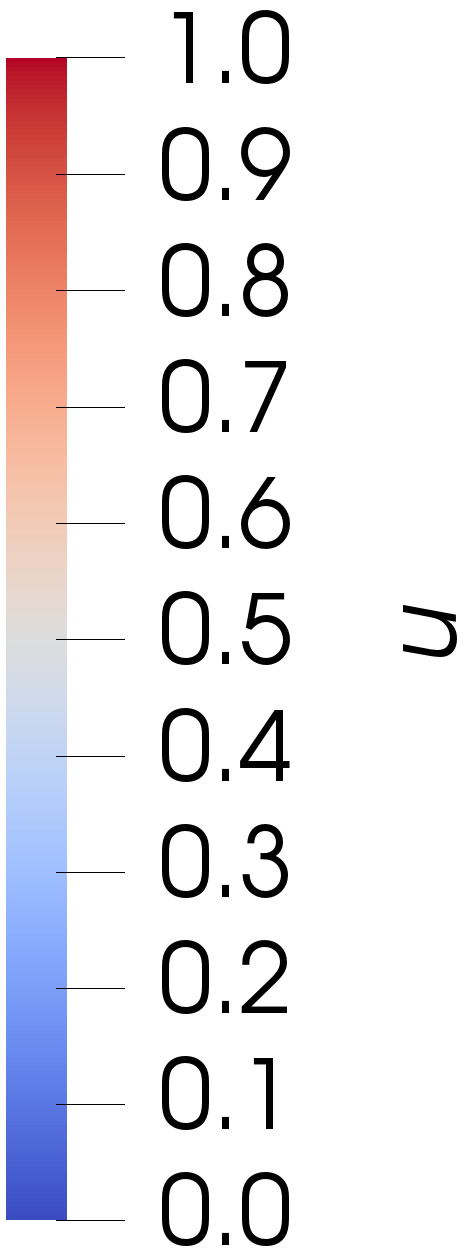}}
\caption{Results for a numerical example using an image with close disks on the unit square. From left to right: given image $g$ (black corresponding to the value $+1$, pale green to $0$, both with $\Phi=1$), stationary state $u$ without disconnectedness penalty, stationary state with disconnectedness penalty. We use $\delta = 140$.}
\label{fig:num1}
\end{center}
\end{figure}

\begin{figure}[ht!]\begin{center}
\raisebox{-0.5\height}{\includegraphics[height=3.7cm]{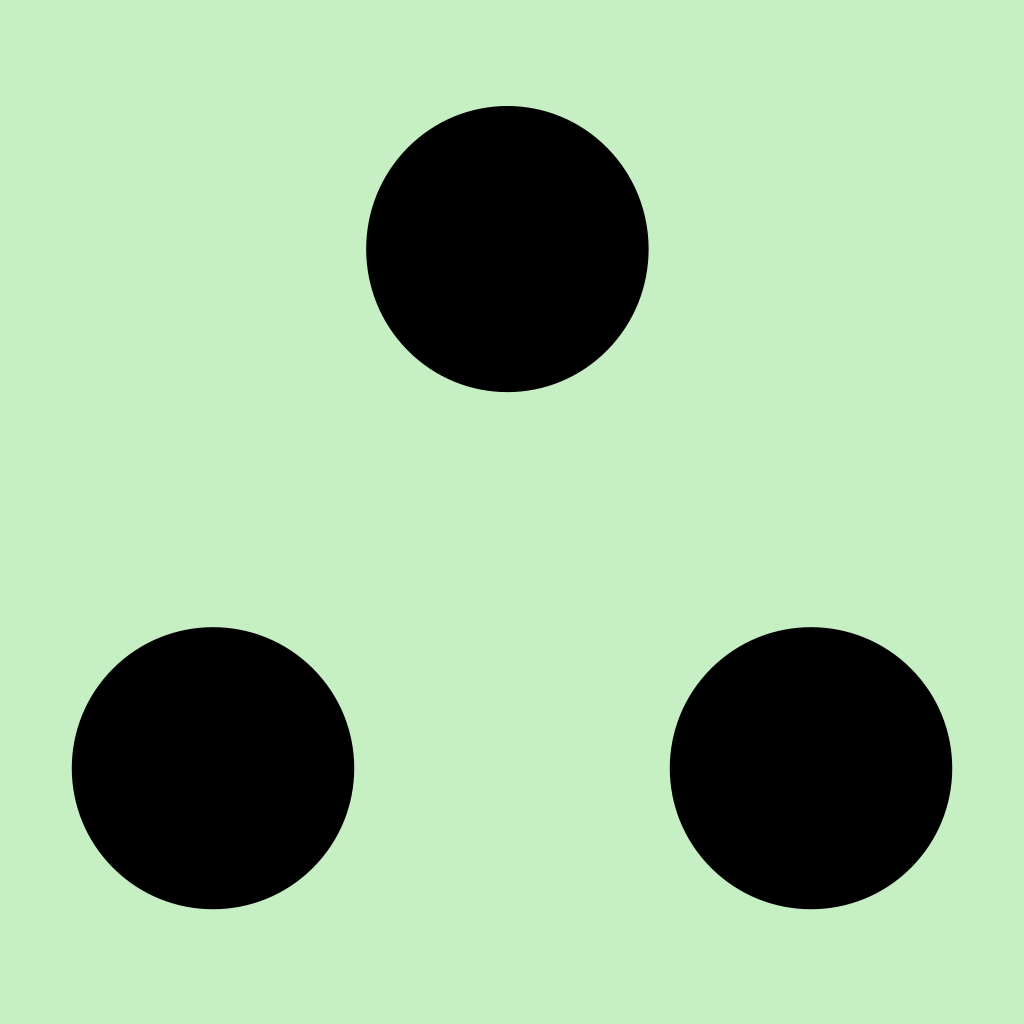}}
\raisebox{-0.5\height}{\includegraphics[height=3.7cm, clip, trim = 36.98cm 18.21cm 36.98cm 18.21cm]{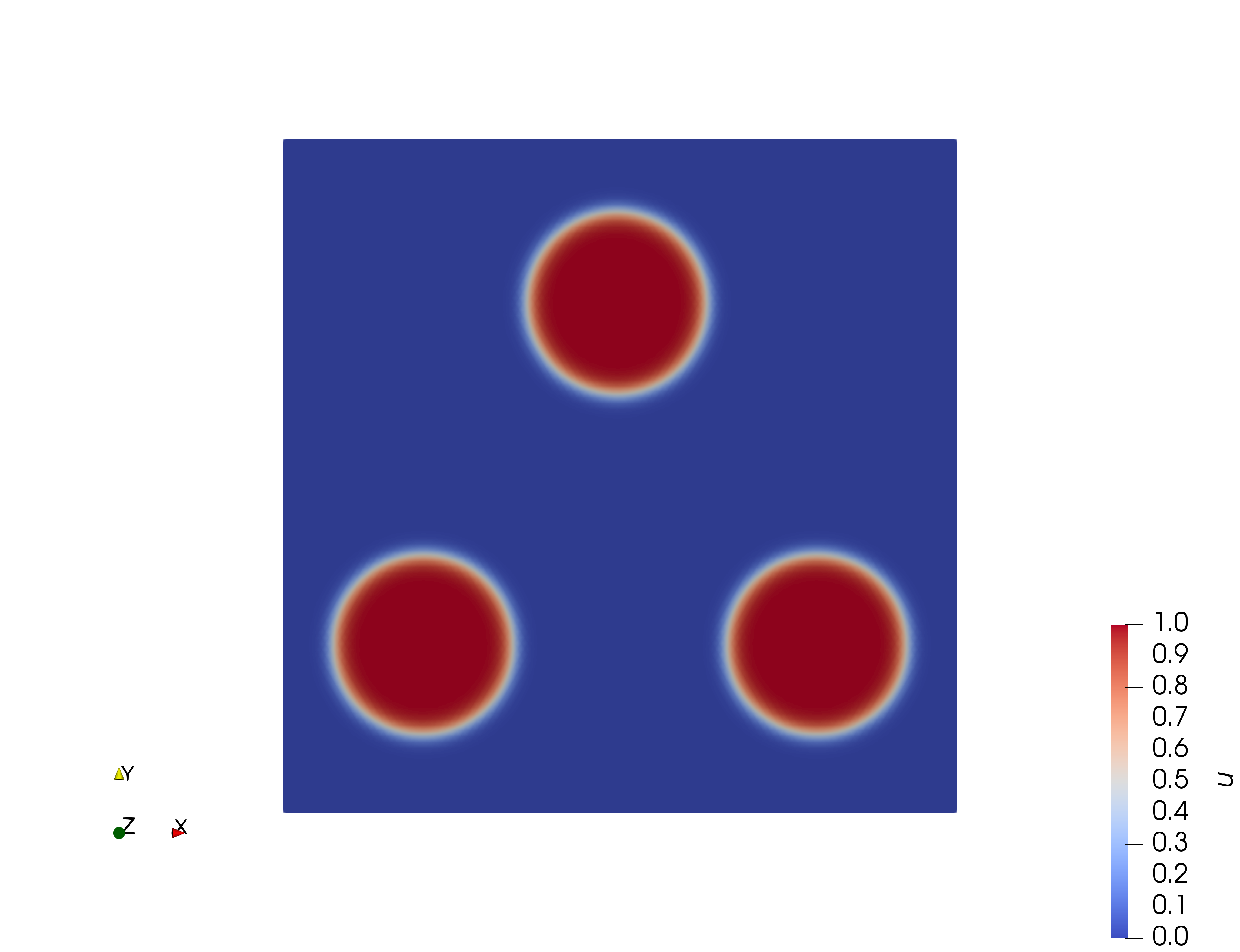}}
\raisebox{-0.5\height}{\includegraphics[height=3.7cm, clip, trim = 36.98cm 18.21cm 36.98cm 18.21cm]{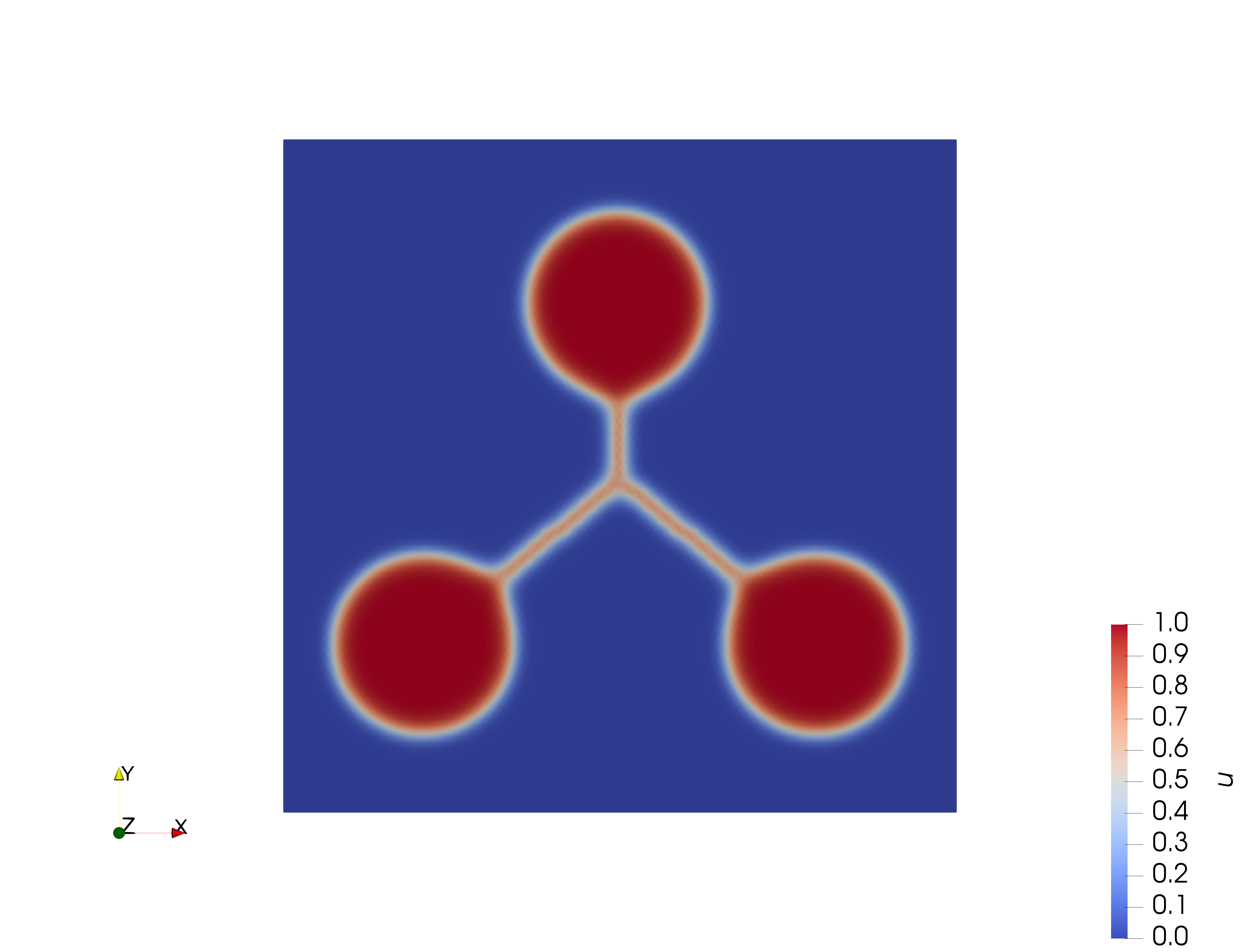}}
\raisebox{-0.5\height}{\includegraphics[height=2.5cm]{colorbar}}
\caption{Results for a numerical example using an image with disks of greater distance, see also Figure~\ref{fig:num1}. We use $\delta = 50$.}
\label{fig:num2}
\end{center}
\end{figure}

The first numerical experiment, illustrated in Figures~\ref{fig:num1} and~\ref{fig:num2}, shows that, indeed, the method produces a phase field approximation of the perimeter of a set plus (twice) its Steiner-tree. The double-layer introduced in order to maintain connectedness is clearly visible. In these simulations, the initial condition was given by $u=1$, interpolated to zero on the boundary. The additional approximate perimeter introduced through the double-layer is $1.3\cdot 10^{-2}$ for the nearby disks, and $0.87$ for the further apart disks. We note that these values are somewhat below the value for twice the length of the connecting double layers ($8.4\cdot 10^{-2}$ and $1.14$, respectively) for the figures, however, our numerical examples were performed with a fairly large value for $\alpha$.

\begin{figure}[ht!]\begin{center}
\raisebox{-0.5\height}{\includegraphics[height=3.7cm]{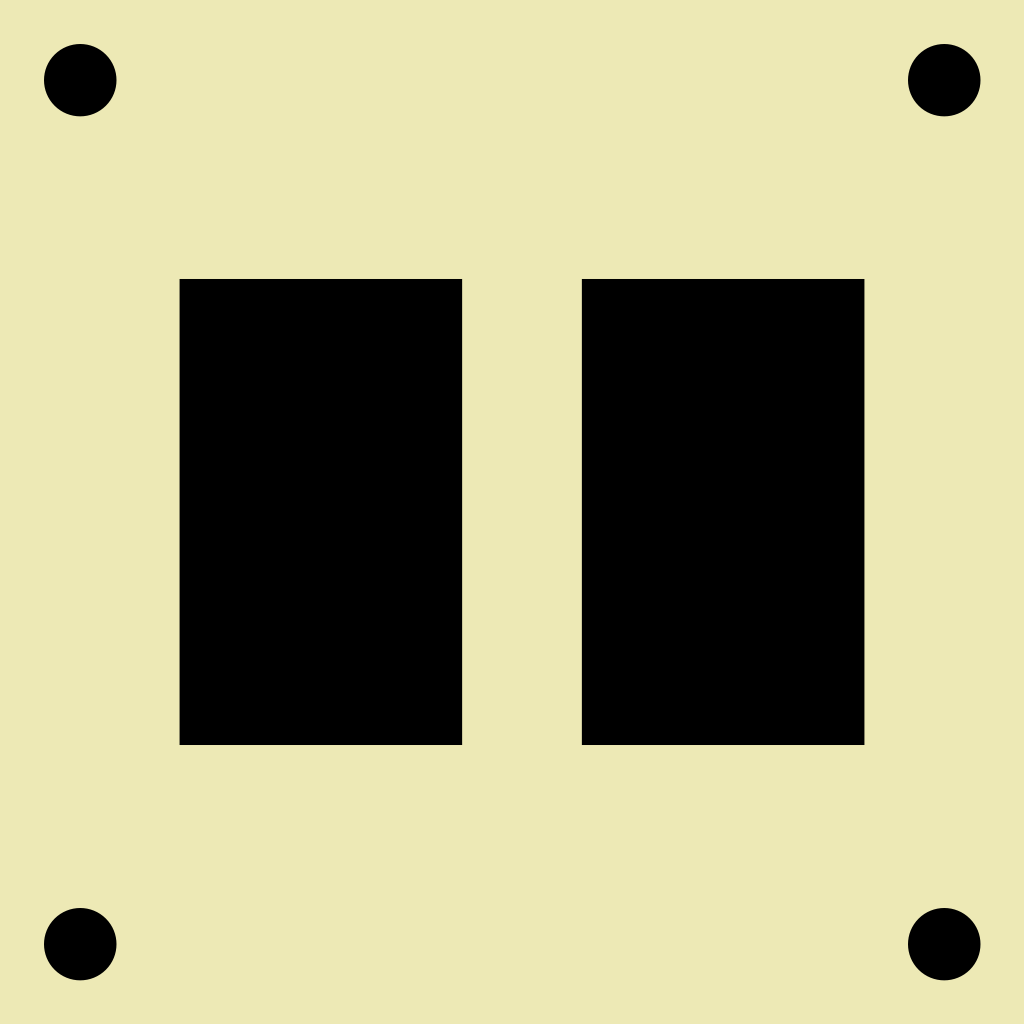}}
\raisebox{-0.0\height}{
\begin{minipage}{3.7cm}\includegraphics[height=3.7cm, clip, trim = 36.98cm 18.21cm 36.98cm 18.21cm]{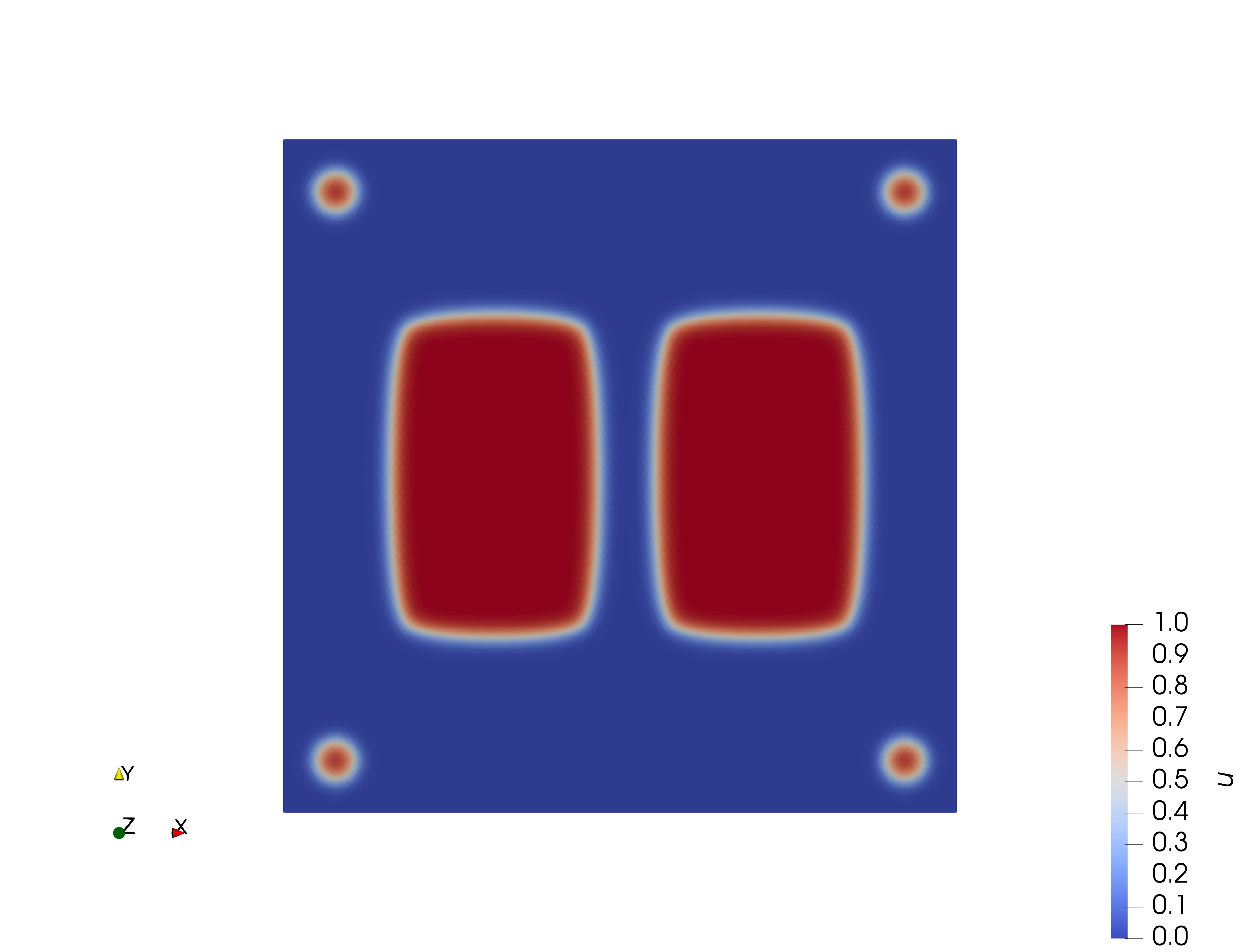} \\[2mm]
\includegraphics[height=3.7cm, clip, trim = 36.98cm 18.21cm 36.98cm 18.21cm]{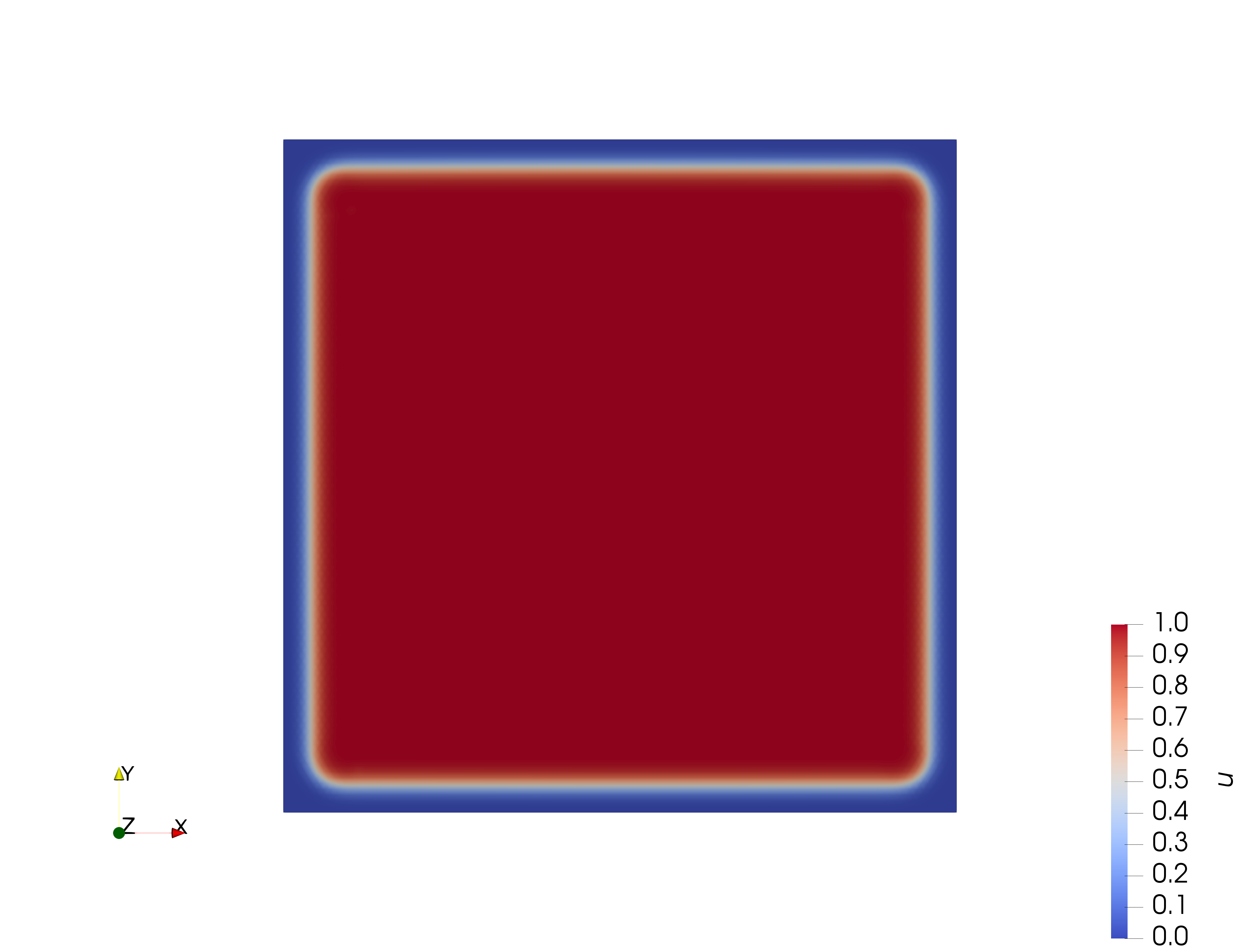}
\end{minipage}
}
\raisebox{-0.5\height}{\includegraphics[height=3.7cm, clip, trim = 36.98cm 18.21cm 36.98cm 18.21cm]{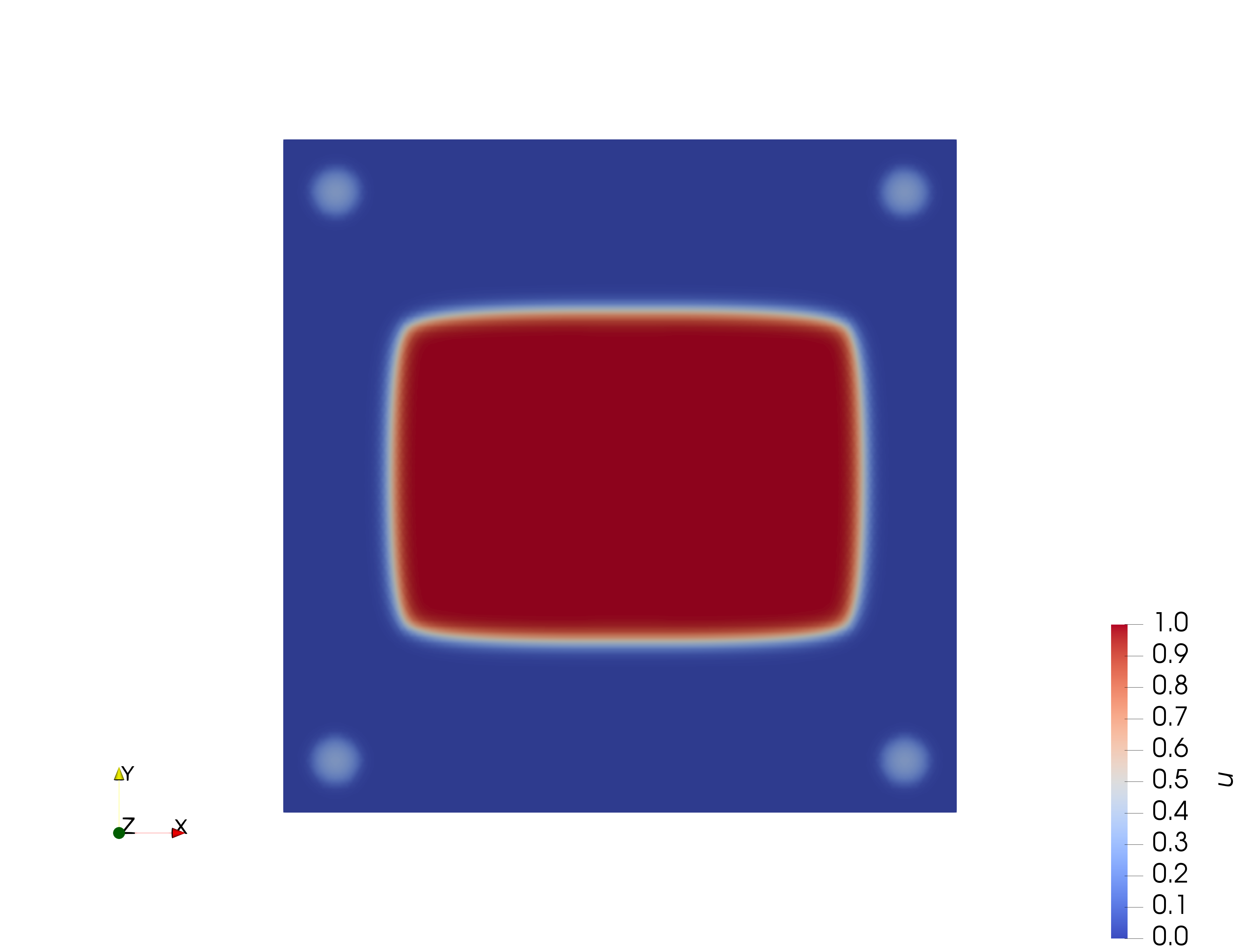}}
\raisebox{-0.5\height}{\includegraphics[height=2.5cm]{colorbar}}
\caption{Results for a numerical example to recover a rectangle that has been partly occluded. The two black rectangles form the image (with the occlusion being the yellow strip between), the four small disks are image artifacts that should be ignored in the segmentation. In the black regions we have $g=\Phi=1$, in the pale yellow region $g=\Phi=0$. The two middle images show the final stationary states without disconnectedness penalty, top for initial condition $u=g$, bottom for $u=1$. The right image is the result of the simulation with disconnectedness penalty and initial condition $u=g$. We use $\delta = 140$.}
\label{fig:num3}
\end{center}
\end{figure}

The second numerical experiment shows the applicability in image segmentation. We would like to recover an a-priori known to be connected object (in this case, a simple rectangle) which has been partly occluded (in this case, by a vertical strip). In addition, there are some smaller artifacts in the image (small disks in our example) that should be ignored. The results of this experiment are displayed in Figure~\ref{fig:num3}. It is evident, that the stationary points in the experiments without topological energy term do not yield the desired recovered image: starting with $u=1$ (as in the first experiment), the mean curvature flow becomes pinned at the obstacles. Starting with $u=g$, however, the four artifacts remain visible in the segmentation and the rectangle is divided into two pieces due to the occlusion. Adding the topological term and starting with $u=g$, however, does yield an approximation of the single rectangle -- the connectedness term creates a bridge between the two pieces which then, in the course of the gradient flow, expands. The artifacts in our experiment are small enough such that the energetically better solution is to pay the fidelity penalty as opposed to creating a connecting double layer.

\section*{Acknowledgements} PWD gratefully acknowledges partial support from the German Scholars Organization/Carl-Zeiss-Stiftung in the form of the Wissenschaftler-R{\"u}ckkehrprogramm.

\appendix

\section{Proof of Lemma \ref{lemma retraction}}

Recall the statement we are showing here: {\em If $U\subset \R^n$ is a convex open set and $K\subset U$ is compact, then there exists a $C^\infty$-diffeomorphism $\phi:\R^n\to U$ such that
\begin{enumerate}
\item $\phi(x) = x$ for all $x\in K$ and
\item $|\phi(x) - \phi(y)|\leq |x-y|$ for all $x, y\in \R^n$.
\end{enumerate}
}

\begin{proof}[Proof of Lemma \ref{lemma retraction}]
{\bf Step 1.} Since $K$ is compact, also its convex hull is compact (easily proved using sequential compactness), and since $U$ is convex, we find that $\conv(K) \subset U$. In particular, $\delta:= \dist(\conv(K), \partial U) >0$. The distance function
\[
d:\R^n \to [0,\infty), \qquad d(x) = \dist(x, \conv(K))
\]
is $1$-Lipschitz and convex since for $x,y\in \R^n$ and $\lambda\in [0,1]$ we have
\begin{align*}
\lambda\,d(x) + (1-\lambda)\,d(y) &= \lambda\,\big|x-\pi(x)\big| + (1-\lambda)\,\big|y-\pi(y)\big|\\
	&= \big|\lambda x-\lambda \,\pi(x)\big| +\,\big|(1-\lambda) y-(1-\lambda)\,\pi(y)\big|\\
	&\geq \big|\lambda x-\lambda \,\pi(x) + (1-\lambda) y-(1-\lambda)\,\pi(y)\big|\\
	&= \big| \big(\lambda x + (1-\lambda)\,y\big) - \big(\lambda\,\pi(x) + (1-\lambda)\pi(y)\big)\big|\\
	&\geq d\left(\lambda x+ (1-\lambda)y\right)
\end{align*}
where $\pi:\R^n\to \conv(K)$ denotes the closest point projection onto a closed convex set $\conv(K)$. Since $d$ is convex, also its convolution $d_\eps$ with a standard mollifier of scale $\eps$ is convex as the convexity property is preserved due to the linearity of the operation. Now, when we choose $\eps$ so small that $\|d- d_\eps\|_{L^\infty} < \frac \delta3$, we can use Sard's theorem and the regular value theorem together with the convexity of $d_\eps$ to find $\frac\delta3 < t < \frac{2\delta}3$ such that $E:= \{ d_\eps< t\}$ satisfies the following:
\begin{enumerate}
\item $\conv(K) \cc E \cc U$,
\item $E$ is convex and
\item $\partial E\in C^\infty$.
\end{enumerate}
We can now forget $K, \pi$ and $\delta$ and only work with $E$.

{\bf Step 2.} Now denote $\delta := \dist(\partial E, \partial U)$, $d_x := \dist(x, \overline E)$ and let $\pi$ be the closest point projection onto $\overline E$. $\pi$ is a $1$-Lipschitz map which is the identity on $\overline E$ and compresses the exterior space $\R^n\setminus \overline E$ into $\partial E\subset \overline E\cc U$. It is $C^\infty$-smooth on $E$ and $\R^n \setminus \overline E$, but only continuous at $\partial E$, and definitely not a diffeomorphism. We can, however, use the little space between $\overline E$ and $\partial U$ to make it smooth and a diffeomorphism.

Let $f:[0,\infty)\to [0, \frac\delta 2)$ be a $C^\infty$-function such that
\[
f(t) = t \quad\forall\ t< \delta/4, \qquad 0< f'\leq 1.
\]
Any point $x\in \R^n \setminus \overline E$ can be written as $x = \pi(x) + d(x)\,\nu_{\pi(x)}$ where $\nu$ is the exterior normal field to $\partial E$ -- it is well-known that the map
\[
\tilde \phi:\partial E\times (-\eps, \eps) \to \{\dist(\cdot, \partial E)<\eps\}, \qquad \tilde \phi(x, t) = x+t\nu_x
\]
is always a diffeomorphism for small enough $\eps$ (see e.g.\ \cite[Section 14.6]{gilbarg:2001vb}), and since $E$ is convex, it is easy to show that the map is bijective on the whole exterior domain using the uniqueness of the closest point projection. It is a diffeomorphism since when $\tau$ is a tangent vector to $\partial E$ we have
\[
\nabla_\tau \tilde\phi(x,t) = \tau + t\,A(\tau), \qquad \partial_t\phi(x,t) = \nu_x
\]
where and $A$ denotes the shape operator of $\partial E$, and because $E$ is convex, $\langle \tau, A(\tau)\rangle\geq0$ and thus the derivative map is injective.
Abbreviating $d_x=d(x)$, we define the new diffeomorphism
\[
\phi: \R^n \to U, \qquad \phi(x) = \begin{cases} x & x\in E\\ \pi(x) + f(d_x)\,\nu_{\pi(x)} &x\notin E.\end{cases}
\]
Since $\phi(x) = x$ for all $x$ in a neighbourhood of $\overline E$, the function $\phi$ is $C^\infty$-smooth. It remains to show that it is a non-expansive diffeomorphism. 

{\bf Step 3.} For diffeomorphic smoothness, we only need to show that
\[
\phi: \R^n\setminus \overline E \to \left\{\dist (\cdot, E) < \frac\delta2\right\}\setminus \overline E
\]
is a diffeomorphism since $\phi=\id$ on a neighbourhood of $\overline E$. But this is obvious since $\phi$ is given as
\[
\phi(x) = \tilde \phi \circ \alpha \circ \tilde\phi^{-1}(x) \qquad\text{where }\alpha:\partial E\times (0, \infty) \to \partial E\times\left(0, \frac\delta2\right), \quad \alpha(x, t) = \big(x, f(t)\big)
\]

{\bf Step 4.} To see that $\phi$ is non-expansive, we need to check this in the cases that $x, y\in \R^n\setminus \overline E$ and $x\in \overline E$, $y\in \R^n\setminus \overline E$. Let us look at the simpler second case first. Then
\begin{align*}
\big|\phi(x) - \phi(y)\big|^2 &= \big| x - \big[ \pi(y) + f(d_y)\nu_y\big]\big|^2\\
	&= \big|x - \pi(y)\big|^2 + 2\,f(d_y)\,\langle x-\pi(y), -\nu_y\rangle + f(d_y)^2\\
	&\leq \big|x - \pi(y)\big|^2 + 2\,d_y\,\langle x-\pi(y), -\nu_y\rangle + d_y^2\\
	&= |x-y|^2
\end{align*}
since $f(d_y) \leq d_y$ and since $x\in \overline E$, so by a common characterization of the closest point projection
\[
\langle -\nu_y, x-\pi(y)\rangle = \frac{1}{|y-\pi(y)|} \,\langle y-\pi(y), \pi(y) - x\rangle \geq 0.
\]
In the first case, we have
\begin{align*}
\big|\phi(x) - \phi(y)\big|^2 &=\big| \pi(x)+ f(d_x)\,\nu_{\pi(x)} - \big[\pi(y) + f(d_y)\,\nu_{\pi(y)}\big]\big|^2\\
	&= \big|\pi(x) - \pi(y)\big|^2 + 2\,\langle \pi(x) - \pi(y), f(d_x) \nu_{\pi(x)} - f(d_y)\,\nu_{\pi(y)}\rangle\\
		&\qquad + \big|f(d_x)\,\nu_{\pi(x)} - f(d_y)\nu_{\pi(y)}\big|^2\\
	&= \big|\pi(x) - \pi(y)\big|^2 + 2\,f(d_x) \langle \nu_{\pi(x)}, \pi(x)-\pi(y)\rangle + 2\,f(d_y)\langle \nu_{\pi(y)}, \pi(y)-\pi(x)\rangle\\
		&\qquad + \big|f(d_x)\,\nu_{\pi(x)} - f(d_y)\nu_{\pi(y)}\big|^2\\
	&\leq \big|\pi(x) - \pi(y)\big|^2 + 2\,d_x \langle \nu_{\pi(x)}, \pi(x)-\pi(y)\rangle + 2\,d_y\langle \nu_{\pi(y)}, \pi(y)-\pi(x)\rangle\\
		&\qquad + \big|f(d_x)\,\nu_{\pi(x)} - f(d_y)\nu_{\pi(y)}\big|^2
\end{align*}
as before. To treat the last term, consider $0<s\leq t$ and $\nu, e\in S^{n-1}$ and observe that
\begin{align*}
\big|f(t)\nu - f(s)\,e\big|^2 &= \big|\,\big[ f(t)-f(s)\big]\nu + f(s)\,[\nu-e]\,\big|^2\\
	&= \big[f(t)-f(s)\big]^2 + 2\big[f(t)-f(s)\big]\,f(s) \,\langle \nu, \nu-e\rangle + f(s)^2|\nu-e|^2\\
	&\leq |t-s|^2 + 2[t-s]\,s\,\langle \nu, \nu-e\rangle + s^2|\nu-e|^2\\
	&= |t\nu-se|^2.
\end{align*}
Applying this in the above inequality, we find that in total
\begin{align*}
\big|\phi(x) - \phi(y)\big|^2 &\leq \big|\pi(x) - \pi(y)\big|^2 + 2\,d_x \langle \nu_{\pi(x)}, \pi(x)-\pi(y)\rangle + 2\,d_y\langle \nu_{\pi(y)}, \pi(y)-\pi(x)\rangle\\
		&\qquad + \big|d_x\,\nu_{\pi(x)} - d_y\nu_{\pi(y)}\big|^2\\
	&= |x-y|^2.
\end{align*}
This concludes the proof.
\end{proof}


\end{document}